\documentclass[11pt]{amsart}

\usepackage[matrix,arrow,curve,frame]{xy}
\usepackage{amsmath,amsthm,amssymb,enumerate}
\usepackage{latexsym}
\usepackage{amscd}
\usepackage[colorlinks=true]{hyperref}
\usepackage{euscript}

\newtheorem{thm}[subsection]{Theorem}
\newtheorem{defn}[subsection]{Definition}
\newtheorem{prop}[subsection]{Proposition}

\newtheorem{corr}[subsection]{Corollary}

\newtheorem{remark}[subsection]{Remark}

\theoremstyle{definition}

\newcommand{\R}{\mathbb R}
\newcommand{\Q}{\mathbb Q}
\newcommand{\Z}{\mathbb Z}

\newcommand{\C}{\mathbb C}
\newcommand{\CP}{\mathbb P}

\newcommand{\ess}{{\mathbb S}}

\DeclareMathOperator{\Hom}{Hom}
\DeclareMathOperator{\Ext}{Ext}
\DeclareMathOperator{\Alg}{Alg}

\DeclareMathOperator{\BO}{BO}
\DeclareMathOperator{\BU}{BU}

\DeclareMathOperator{\BtUSpin}{B(\tilde{U}/Spin)}
\DeclareMathOperator{\U}{U}

\DeclareMathOperator{\Spin}{Spin}
\DeclareMathOperator{\SO}{SO}
\DeclareMathOperator{\BUO}{B(U/O)}

\DeclareMathOperator{\UO}{U/O}
\DeclareMathOperator{\USO}{U/SO}
\DeclareMathOperator{\BUSO}{B(U/SO)}
\DeclareMathOperator{\SUSO}{{SU/SO}}

\DeclareMathOperator{\tU}{\tilde{U}}
\DeclareMathOperator{\Or}{O}

\DeclareMathOperator{\MSO}{MSO}
\DeclareMathOperator{\MSpin}{MSpin}
\DeclareMathOperator{\SpU}{Sp/U}
\DeclareMathOperator{\KU}{KU}
\DeclareMathOperator{\sH}{H}

\DeclareMathOperator{\E}{E}

\DeclareMathOperator{\HC}{H(\C)}
\DeclareMathOperator{\Hrt}{H(\Z/2)}

\DeclareMathOperator{\Exp}{Exp}
\DeclareMathOperator{\De}{Q}
\DeclareMathOperator{\THH}{THH}

\DeclareMathOperator{\Sph}{S}

\DeclareMathOperator{\Gr}{G}

\DeclareMathOperator{\MU}{MU}
\DeclareMathOperator{\MSp}{MSp}
\DeclareMathOperator{\MtU}{M\tilde{U}}

\DeclareMathOperator{\HHh}{HH}

\newfont{\german}{eufm10}

 \DeclareMathOperator{\End}{End}

\newcommand\qu{/\kern-.7ex/}

\newcommand{\oh}{{\mathcal O}}

\newcommand{\ab}{{\rm ab}} 
\newcommand{\kO}{{\rm kO}}

\newcommand{\Qsymm}{{\sf QSymm}}
\newcommand{\MTM}{{\sf MTM}}
\newcommand{\Symm}{{\sf Symm}}
\newcommand{\MZV}{{\rm MZV}}
\newcommand{\Nsymm}{{\sf NSymm}}
\newcommand{\GT}{{\sf GT}}
\newcommand{\Gm}{{{\mathbb G}_m}}
\newcommand{\alg}{{\rm alg}}
\newcommand{\bzeta}{{\boldsymbol{\zeta}}}

 \newcommand{\mF}{{\mathfrak F}}
 \newcommand{\mf}{{\mathfrak f}}

\begin{document}
\pagestyle{plain}

\title
{The Stable Symplectic Category and a conjecture of Kontsevich}
\author{Nitu Kitchloo}
\author{Jack Morava}
\address{Department of Mathematics, Johns Hopkins University, Baltimore, USA}
\email{nitu@math.jhu.edu}
\email{jack@math.jhu.edu}
\thanks{Nitu Kitchloo is supported in part by NSF through grant DMS
  1005391.}

\date{\today}

{\abstract
We consider an oriented version of the stable symplectic category defined in \cite{N}. We show that the group of monoidal automorphisms of this category, that fix each object, contains a natural subgroup isomorphic to the solvable quotient (or a graded-abelian quotient) of the Grothendieck--Teichm\"uller group. This establishes a stable version of a conjecture of Kontsevich which states that groups closely related to the Grothendieck--Teichm\"uller group act on the moduli space of certain field theories \cite{KO}. The above quotient is also shown to be the motivic group of monoidal automorphisms of a canonical representation (or fiber functor) on the stable symplectic category.}
\maketitle

\tableofcontents
\section{Introduction}

\noindent
In his solution to the existence of deformation quantization of a Poisson manifold $(M, \omega)$, Kontsevich showed in \cite{K1} that there exists a quasi-equivalence of $L_\infty$-algebras:
\[ \mbox{dQ} :  \Gamma(M, \Lambda^\ast(M)) \, [[ \hbar ]] \longrightarrow \mbox{Hoch}^\ast(\mathcal{O}(M)) \, [[ \hbar ]],  \, \,  \quad \hbar \, \omega \leftrightarrow [\mathcal{O}_\hbar(M)], \] 
where the left hand side is the DGLA of poly-vector fields (with $\mbox{d}=0$) that controls Poisson structures on $M$, and the right hand side is the DGLA of poly-differential operators on $M$ that controls the associative deformations of the functions $\mathcal{O}(M)$ on $M$. Hence, up to equivalence, the Poisson structure $\omega$ yields a unique associative deformation $\mathcal{O}_\hbar(M)$. However, the map $\mbox{dQ}$ is not unique.  The set of all such equivalences is a torsor under the action of a certain pro-solvable, pro-algebraic group $\GT_\Q$ over $\Q$, known as the Grothendieck-Teichm\"uller group (see \S \ref{hotGT}). This group is a pro-solvable, pro-algebraic group with a Lie algebra given by the extension of a free Lie algebra on generators: $z_{3}, z_{5}, \ldots$, by a grading generator $z_0$:
\[ \mbox{Lie} \, \GT_\Q = \Q \langle z_0 \rangle \ltimes \langle z_3, z_5, \ldots, z_{2k+1}, \ldots \rangle, \quad \quad [z_0, z_{2k+1}] = (2k+1) z_{2k+1}. \]

\medskip
\noindent
 The above result suggests that if $(M,\omega)$ is the space of classical solutions of a given classical field theory, then the group $\GT_\Q$ acts on the collection of all possible rings of quantum observables: $\mathcal{O}_\hbar(M)$. Based on this  observation, and explicit computations of certain Feynman path integrals, Kontsevich conjectured \cite{KO} (Section 5) that groups closely related to the Grothendieck-Teich\"muller group act on the Moduli space of Quantum Field theories.

\medskip
\noindent
We will approach Kontsevich's conjecture from the standpoint of geometric quantization. As opposed to deformation quantization, the method of geometric quantization is an attempt to describe a quantum field theory directly and functorially with respect to the space of classical solutions. Functoriality above is to be understood as asking for the existence of a geometric quantization functor with values in the category of topological vector spaces:
\[ \mbox{gQ} : \mathcal{S} \longrightarrow \mbox{Top. V.S}, \]
where $\mathcal{S}$ is a suitable category of correspondences between symplectic manifolds known as Weinstein's Symplectic Category \cite{W,W2}. 
In particular, $\mbox{gQ}$ applied to the Poisson manifold of classical solutions $(M, \omega)$ in $\mathcal{S}$ should recover the quantum state space of the field theory in question. 

\medskip
\noindent
Notice that if we could construct an action of $\GT_\Q$ on $\mathcal{S}$ that fixes each object, we would obtain an action of the Grothendieck-Teich\"muller group (by pre-composition) on the set of field theories as suggested by the conjecture of Kontsevich above. It is our goal to make this a factual statement. However, in order to achieve that goal, we will need to first stabilize $\mathcal{S}$. As we shall see in a moment, there are problems with the definition of $\mathcal{S}$ as stated. 

\medskip
\noindent
In \cite{N} the first author defined a stabilization $h\ess$ of the symplectic category $\mathcal{S}$ introduced by A. Weinstein in the 1980's \cite{W,W2}. The objects of Weinstein's category are symplectic manifolds, and the morphisms between two symplectic manifolds $(M,\omega)$ and $(N,\eta)$ are lagrangian immersions into $\overline{M} \times N$, where the conjugate symplectic manifold $\overline{M}$ is defined by the pair $(M,-\omega)$. The composition $L_1 \ast L_2$ of two lagrangian immersions $L_1 \looparrowright \overline{M} \times N$ and $L_2 \looparrowright \overline{N} \times K$, is defined to be the fiber product: $L_1 \times_N L_2 \longrightarrow \overline{M} \times K$. This definition does not always yield a lagrangian immersion to $\overline{M} \times K$: to do so, the pullback defining it must be transverse, so Weinstein's construction is unfortunately not a genuine category. \medskip

\noindent
In \cite{N}, we described a way to extend the symplectic category to an honest category $h\ess$, by introducing a moduli space of stabilized (in the sense of homotopy theory) lagrangian immersions in a symplectic manifold of the form $\overline{M} \times N$ \footnote{under the assumption of monotonicity. Otherwise, one has the space of totally real immersions.}. This moduli space can be described as the infinite loop space corresponding to a certain Thom spectrum. Taking this as the space of morphisms defines a {\em stable symplectic (homotopy) category} $h\ess$ that is naturally enriched over the homotopy category of spectra (under smash product) \footnote{In \cite{N}, we lifted this category to an $A_\infty$-category $\ess$ enriched over spectra.}. Composition in $h\ess$ is well-defined and remains faithful to Weinstein's original definition. Geometrically, the stabilization of Weinstein's category can be seen as ``inverting the symplectic manifold $\C$", analogous to the introduction of an inverse to the projective line in the theory
of motives. In other words, we introduce a relation on the symplectic category that identifies two symplectic manifolds $M$ and $N$ if $M \times \C^k$ becomes equivalent to $N \times \C^k$ for some $k$. Note that the ring of observables corresponding to $\C$ is the (simple) algebra known as the Heisenberg algebra corresponding to a single free Boson. So in a suitable sense the stable symplectic category can be seen as the symplectic analog of the ``derived rings of observables" studied by G. Segal in \cite{GS}. The stable symplectic category also has variants defined by lagrangian immersions endowed with orientations or metaplectic structures (see \cite{N})

\medskip
\noindent
Here we study the group of monoidal automorphisms of the {\bf oriented} stable symplectic homotopy category: $hs\ess$. In other words, by extending coefficients to commutative algebras $\E$ over $s\Omega$, we consider a family of categories $hs\ess \wedge_{s\Omega} \E$. Consequently, we have the family of groups that represent monoidal automorphisms of the categories $hs\ess \wedge_{s\Omega} \E$ that fix each object. We describe this group (see corollary \ref{expo} and theorem \ref{thh}), and relate it in characteristic zero to an abelian quotient of the Gothendieck--Teichm\"uller group \cite{KO}. This establishes a stable version of a conjecture of Kontsevich which states that the Grothendieck--Teichm\"uller group acts on the moduli space of certain field theories (see \cite{KO}, section 5). The above abelian quotient is also shown to be the motivic group of monoidal automorphisms of a canonical representation (or fiber functor) on a truncation of the oriented stable symplectic category. The value of this functor on an object $M$ is the $s\Omega$-module $s\Omega(M)$ representing the space of stably immersed oriented lagrangians in $M$. 

\medskip
\noindent
This document is organized as follows: In section \ref{hcat} we recall the construction of the stable symplectic homotopy category. In section \ref{Kontsevich} we study the group of monoidal automorphisms of the oriented stable symplectic homotopy category over the rational numbers which fix each object. We identify this group with a graded abelian quotient of the Grothendieck--Teichm\"muller group and reinterpret it as validation of a stable conjecture of Kontsevich on the action of the Grothendieck-Teich\"muller group on the moduli space of field theories. In this section we also identify the above quotient as the motivic group of monoidal automorphisms of the fiber functor described above. Section \ref{Kontsevich} also contains a description of an integral model for the group of monoidal automorphisms of the oriented stable symplectic homotopy category. The Appendix summarizes some properties of the Grothendieck-Teichm\"uller group and related constructions.  In particular, we define a Hopf-Galois analog of the Grothendieck-Teichm\"uller group in the category of spectra, and identify its abelianization with the group of symmetries identified in Section \ref{Kontsevich}. 

\medskip
\noindent
Before we begin, we would like to thank M. Abouzaid, V. Angelveit, A. Baker, Y. Eliashberg, D. Gepner, M. Hazewinkel, J. Lind, B. Richter and H. L. Tanaka for helpful conversations related to this project.

\section{The Stable Symplectic homotopy category} \label{hcat}

\noindent
In this section we recall the construction of the stable symplectic homotopy category \cite{N}. Given a symplectic manifold $(M, \omega)$ of real dimension $2m$, we construct a spectrum $\Omega(M)$ so that the corresponding infinite loop space can be interpreted as a space whose points represent manifolds that admit totally-real immersions into $M \times \C^n$ for large values of $n$ (up to an equivalence that shall be made precise later). We will say that $M$ satisfies monotonicity if the cohomology class of the symplectic form $\omega$ is a scalar multiple of the first Chern class of $M$. Under the assumption of monotonicity, totally real immersions may be replaced by lagrangian immersions in the above interpretation of $\Omega(M)$. 


\noindent
{\bf \S 2.1 The basic construction}

\medskip
\noindent
Consider the Thom spectrum $\Sigma^n {\mathcal{G}}(\tau \oplus \C^n)^{-\zeta_n}$, where the bundle $\zeta_n$ is defined by virtue of the pullback diagram: 
\[
\xymatrix{
{\mathcal{G}}(\tau \oplus \C^n)       \ar[d] \ar[r]^{\zeta_n \quad} & \BO(m+n) \ar[d] \\
M      \ar[r]^{\tau \oplus \C^n \quad \quad} & \BU(m+n),
}
\]
where $\tau$ denotes (homotopy unique) complex structure on the tangent bundle of $M$ compatible with the symplectic form $\omega$. In \cite{N}, we used the work of D. Ayala \cite{A} to show that for $n>0$, the infinite loop space $\Omega^{\infty-n} ({\mathcal{G}}(\tau \oplus \C^n)^{-\zeta_n})$ can be interpreted as the moduli space of manifolds $L^{m+n} \subset \R^{\infty} \times \R^n$, with a proper projection onto $\R^n$, and endowed with a totally-real immersion $L^{m+k} \looparrowright M \times \C^n$ (or lagrangian immersion, under the assumption of monotonicity). More precisely, the space $\Omega^{\infty-n} ({\mathcal{G}}(\tau \oplus \C^n)^{-\zeta_n})$ is uniquely defined by the property that given a smooth manifold $X$, the set of homotopy classes of maps $[X, \Omega^{\infty-n} ({\mathcal{G}}(\tau \oplus \C^n)^{-\zeta_n})]$, is in bijection with concordance classes of smooth manifolds $E \subset X \times \R^\infty \times \R^n$ over $X$, so that the first factor projection: $\pi : E \longrightarrow X$ is a submersion, and which are endowed with a smooth map $\varphi : E \longrightarrow M \times \C^n$ which restricts to a totally-real immersion (resp. lagrangian) on each fiber of $\pi$. As before, we demand that the third factor projection $E \longrightarrow \R^n$ be fiberwise proper over $X$. 

\smallskip
\noindent
Now the standard inclusion $\R^{n_1} \subseteq \R^{n_2}$, induces a natural map: 
\[ \varphi_{n_1,n_2} : \Sigma^{n_1} {\mathcal{G}}(\tau \oplus \C^{n_1})^{-\zeta_{n_1}} \longrightarrow \Sigma^{n_2} {\mathcal{G}}(\tau \oplus \C^{n_2})^{-\zeta_{n_2}}, \] 
which represents the map that sends a concordance class $E$, to $E \times \R^{n_2-n_1}$, by simply taking the product with the orthogonal complement of $\R^{n_1}$ in $\R^{n_2}$.

\begin{defn} \label{stablag}
Define the Thom spectrum $\Omega(M)$ representing the infinite loop space of stabilized totally-real (resp. lagrangian under the assumption of monotonicity) immersions in $M$ to be the colimit:
\[ \Omega(M) = \underline{\mathcal{G}}(M)^{-\zeta} := \mbox{colim}_n \, \, \Sigma^n {\mathcal{G}}(\tau \oplus \C^n)^{-\zeta_n}. \]
 Notice that by definition, we have a canonical homotopy equivalence: $\Omega(M \times \C) \simeq \Sigma^{-1} \Omega(M)$. Taking $M$ to be a point, we define $\Omega = \Omega(\ast) = (\UO)^{-\zeta}$, where the bundle $\zeta$ over $\UO$ is the virtual zero dimensional bundle over $(\UO)$ defined by the canonical inclusion $\UO \longrightarrow \BO$. \end{defn}

\smallskip

\noindent
Henceforth, we shall use the term ``lagrangian immersion" to mean ``totally-real immersion" if the condition of monotonicity fails to hold. We take this opportunity to also introduce the (abusive) convention of not decorating the stable vector bundle $\zeta$ by the underlying manifold $M$. Hopefully, the manifold $M$ will be clear from context. \medskip

\noindent
We may also describe $\Omega(M)$ as a Thom spectrum: Let the stable tangent bundle of $M$ of virtual (complex) dimension $m$ be given by a map $\tau : M \longrightarrow \Z \times \BU$. As the notation suggests, let $\underline{\mathcal{G}}(M)$ be defined as the pullback: 
\[
\xymatrix{
\underline{\mathcal{G}}(M)       \ar[d] \ar[r]^{\zeta \, \,  } & \Z \times \BO \ar[d] \\
M      \ar[r]^{\tau \quad} & \Z \times \BU.
}
\]
Then the spectrum $\Omega(M)$ is homotopy equivalent to the Thom spectrum of the stable vector bundle $-\zeta$ over $\underline{\mathcal{G}}(M)$ defined in the diagram above. 

\smallskip
\noindent
Notice that the fibration $\Z \times \BO \longrightarrow \Z \times \BU$ is a principal bundle up to homotopy, with fiber being the infinite loop space $\UO$. Hence, the spectrum $\Omega(M)$ is homotopy equivalent to a $\Omega$-module spectrum. Now, observe that we have the equivalence, up to homotopy, of $\UO$-spaces:
\[ \underline{\mathcal{G}}(M) \times_{\UO} \underline{\mathcal{G}}(N) \simeq \underline{\mathcal{G}}(M \times N). \]
This translates to a canonical homotopy equivalence:
\[ \mu : \Omega(M) \wedge_\Omega \Omega(N) \simeq  \Omega(M \times N). \]

\smallskip
\noindent
Let us now describe the stable symplectic homotopy category $h\mathbb{S}$. The objects of this category will be symplectic manifolds $(M,\omega)$ (see remark \ref{noncpct}), endowed with a compatible almost complex structure.  

\begin{defn}
The spectrum $\Omega(M,N)$ of morphisms in $h\mathbb{S}$ from $M$ to $N$ is the $\Omega$-module spectrum: 
\[ \Omega (M,N) := \Omega(\overline{M} \times N). \]
\end{defn}

\begin{remark}\label{noncpct}
Notice that objects in $h\mathbb{S}$ need not be compact. The price we pay for this, familiar from other contexts, is that we simply lose the identity morphisms for non-compact objects.  
\end{remark}

\noindent
The next step is to define composition. The simplest case 
\[ \Omega(M, \ast) \wedge_{\Omega} \Omega(\ast, N) \longrightarrow \Omega(M,N), \]
is the map $\mu$ constructed earlier. For the general case, consider $k+1$ objects objects $M_i$ with $0 \leq i \leq k$, and let the space $\underline{\mathcal{G}}(\Delta)$  be defined by the pullback:
\[
\xymatrix{
\underline{\mathcal{G}}(\Delta)       \ar[d]^{\xi} \ar[r] &  \underline{\mathcal{G}}(\overline{M}_0 \times M_1 \times \cdots  \times \overline{M}_{k-1} \times M_k) \ar[d] \\
\overline{M}_0 \times (M_1 \times \cdots \times M_{k-1})  \times M_k    \ar[r]^{\Delta \quad \quad \quad \quad} &  \overline{M}_0 \times (M_1 \times \overline{M}_1) \times \cdots \times (M_{k-1} \times \overline{M}_{k-1}) \times M_k 
}
\]
where $\Delta$ denotes the product to the diagonals $\Delta : M_i \longrightarrow M_i \times \overline{M}_i$, for $0 < i < k$. 

\medskip
\noindent
Now notice that the fibrations defining the pullback above are direct limits of smooth fibrations with compact fiber. Furthermore, the map $\Delta$ is a proper map for any choice of $k+1$-objects (even if they are non-compact). In particular, we may construct the Pontrjagin--Thom collapse map along the top horizontal map by defining it as a direct limit of Pontrjagin--Thom collapses for each stage. 

\smallskip
\noindent
Let $\zeta_i$ denote the individual structure maps $\underline{\mathcal{G}}(\overline{M} _{i-1} \times M_i) \longrightarrow \Z \times \BO$, and let $\eta(\Delta)$ denote the normal bundle of $\Delta$. Performing the Pontrjagin--Thom construction along the top horizontal map in the above diagram yields a morphism of spectra:
\[ \varphi : \Omega(M_0, M_1) \wedge_\Omega \cdots \wedge_\Omega \Omega(M_{k-1}, M_k) \simeq \Omega(\overline{M}_0 \times M_1 \times \cdots  \times \overline{M}_{k-1} \times M_k) \longrightarrow \underline{\mathcal{G}}(\Delta)^{-\lambda} \]
where $\lambda : \underline{\mathcal{G}}(\Delta)  \longrightarrow \Z \times \BO$ is the formal difference of the bundle $\bigoplus \zeta_i$ and the pullback bundle $\xi^* \eta(\Delta)$.

\medskip
\noindent
The next step in defining composition is to show that $ \underline{\mathcal{G}}(\Delta)^{-\lambda}$ is canonically homotopy equivalent to $\Omega(M_0, M_k) \wedge (M_1 \times \cdots \times M_{k-1})_+$, where $(M_1 \times \cdots \times M_{k-1})_+$ denotes the manifold $M_1 \times \cdots \times M_{k-1}$ with a disjoint basepoint. To achieve this, it is sufficient to construct a $\UO$-equivariant map over $\overline{M}_0 \times (M_1 \times \cdots \times M_{k-1})  \times M_k$:
\[ \psi : \underline{\mathcal{G}}(\overline{M}_0 \times M_k) \times (M_1 \times \cdots \times M_{k-1}) \longrightarrow \underline{\mathcal{G}}(\Delta), \]
that pulls $\lambda$ back to the bundle $\zeta \times 0$. The construction of $\psi$ is straightforward. We define:
\[ \psi( \lambda, m_1, \ldots, m_{k-1}) = \lambda \oplus \Delta(T_{m_1}(M_1)) \oplus \cdots \oplus \Delta(T_{m_{k-1}}(M_{k-1})), \]
where $\Delta(T_m(M)) \subset T_{(m,m)}(M \times \overline{M})$ denotes the diagonal lagrangian subspace. Now let $\pi : \underline{\mathcal{G}}(\Delta)^{-\lambda} \longrightarrow \Omega(M_0,M_k)$ be the projection map that collapses $M_1 \times \cdots \times M_{k-1}$ to a point. 

\begin{defn}
We define the composition map to be the induced composite:
\[ \pi \varphi :  \Omega(M_0, M_1) \wedge_\Omega \cdots \wedge_\Omega \Omega(M_{k-1}, M_k) \longrightarrow \underline{\mathcal{G}}(\Delta)^{-\lambda} \longrightarrow \Omega(M_0,M_k). \]
We leave it to the reader to check that composition as defined above is homotopy associative. 
\end{defn}

\noindent
{\bf \S 2.2 The identity morphism:}

\medskip
\noindent
We now show that a compact manifold $(M, \omega)$ has an identity morphism: 
\begin{prop}
Let $M$ be a compact manifold, and let $[id] : \Sph \longrightarrow \Omega(M,M)$ be a representative of homotopy class of the diagonal (lagrangian) embedding $\Delta : M \longrightarrow \overline{M} \times M$. Then $[id]$ is indeed the identity for the composition defined above. 
\end{prop}
\begin{proof}
Given two manifolds $M,N$, let $\Delta(M) \subset \overline{M} \times M$ be a diagonal representative of $[id]$ as above. Observe that $\overline{N} \times \Delta(M) \times M$ is transverse to $\overline{N} \times M \times \Delta(M)$ inside $\overline{N} \times M \times \overline{M} \times M$. They intersect along $\overline{N} \times \Delta_3(M)$, where $\Delta_3(M) \subset M \times \overline{M} \times M$ is the triple (thin) diagonal. Hence we get a diagram 
\[
\xymatrix{
\Omega(N, M) \wedge \Sph \ar[dr] \ar[r] & \Omega(N, M) \wedge \Delta(M)^{-\tau} \ar[r]^{\Delta^{-\tau}} \ar[d] &  \Omega(N,M) \wedge_\Omega \Omega(M,M) \ar[d] \\
 & \Omega(N, M) \ar[r]^{=} & \Omega(N,M)}
\]
commutative up to homotopy, where the right vertical map is composition, and the left vertical map is the Pontrjagin--Thom collapse over the inclusion map 
\[ \overline{N} \times M = \overline{N} \times \Delta_3(M) \longrightarrow \overline{N} \times M \times \Delta(M). \]
Now consider the following factorization of the identity map: 
\[ \overline{N} \times M = \overline{N} \times \Delta_3(M) \longrightarrow \overline{N} \times M \times \Delta(M) \longrightarrow \overline{N} \times M \]
where the last map is the projection onto the first two factors. Performing the Pontrjagin--Thom collapse over this composite shows that 
\[ \Omega(N,M) \wedge \Sph \longrightarrow \Omega(N,M) \wedge \Delta(M)^{-\tau} \longrightarrow \Omega(N,M). \]
is the identity. It follows that right multiplication by $[id] : \Sph \rightarrow \Omega(M,M)$ induces the identity map on $\Omega(N,M)$, up to homotopy. A similar argument works for left multiplication. 
\end{proof}

\begin{remark} \label{cancomp}
Recall that given arbitrary symplectic manifolds $M$ and $N$, the composition map:
\[ \Omega(M,\ast) \wedge_{\Omega} \Omega(\ast, N) \simeq \Omega(M,N) \]
induces a natural decomposition of $\Omega(M,N)$. In particular, arbitrary compositions can be canonically factored using this decomposition, and can be computed by applying the ``inner product'' 
\[ \Omega(\ast, N) \wedge_{\Omega} \Omega(N, \ast) \longrightarrow \Omega. 
\]
to the factors. When $M$ is compact, this pairing defines an equivalence
\[ \Omega(\overline{M}) \longrightarrow \Hom_{\Omega}(\Omega(M), \Omega). \]

\end{remark} 

\noindent
It will also be important below that $h\mathbb{S}$ is a symmetric-monoidal category, with a product given by the cartesian product of symplectic manifolds. \medskip

\noindent
\begin{defn}{The construction of the {\bf oriented} stable symplectic homotopy category $s\ess$ is completely analogous, but with $\Or$ replaced by $\SO$; the commutative ring spectrum $s\Omega = (\USO)^{-\zeta}$ defines its coefficients.}
\end{defn}

\noindent
Now $\Omega$ is an (Eilenberg--MacLane) $\Hrt$-algebra, so we can regard $\ess$ as a category with morphism objects enriched over a classical differential graded algebra. This is not the case for $s\Omega$, but its rationalization $s\Omega \otimes \Q$ is again a generalized Eilenberg--MacLane spectrum, with
\[
s\Omega_* \otimes \Q = \Lambda_\Q[y_{4i+1}, \; i > 0]
\]
an exterior algebra on certain odd - degree generators. Moreover, the category $s\ess$ simplifies considerably when rationalized. In particular, the Thom isomorphism
\[
s\Omega(M)_* \otimes \Q  \cong \sH_*(M,s\Omega_* \otimes \Q)
\]
identifies $s\ess \otimes \Q$ with an (Arnol'd-H\"ormander-Maslov\dots) category of symplectic manifolds, whose morphisms are classical cohomological correspondences with compact support, but with coefficients in the graded ring $s\Omega_* \otimes \Q \cong \sH_*(\USO_+,\Q)$. \medskip

\section{A stable version of Kontsevich's conjecture} \label{Kontsevich}

\medskip
\noindent
As mentioned in the introduction, based on explicit computations, Kontsevich conjectures in \cite{KO} (section 5) that the motivic Galois group acts on the ``moduli space" of free field theories, and that in even dimensions, this action factors through the Grothendieck--Teichm\"uller group. Kontsevich was led to this conjecture by his work on the question of uniqueness of the deformation quantization functor. However, we plan to approach his conjecture from the standpoint of geometric quantization. Note that the moduli space of field theories is not a well defined object and so the above conjecture can only be formulated once we have an appropriate definition in place. 

\medskip
\noindent
Now an $(n+1)$-dimensional classical field theories that occur in physics is typically described by an action functional. The (conjectural) framework of quantization proceeds as follows: By taking solutions of the Euler-Lagrange equations about the germ of an $n$-manifold, one obtains a functor from the $(n+1)$-dimensional cobordism category to a variant of the symplectic category. \footnote{the infinite dimensional nature of symmetries force us to also consider infinite dimensional symplectic manifolds} Then the corresponding quantum field theory is constructed as the composite of the above functor with the geometric quantization functor, the latter having the symplectic category as its domain and taking values in the category of topological vector spaces. An automorphism of the symplectic category that fixes every object would therefore induce a deformation of any such field theory. 

\medskip
\noindent
In this section, we will provide a rigorous framework for Kontsevich's conjecture: we will compute a canonical subgroup of the group of monoidal automorphisms of the oriented {\em stable} symplectic (homotopy) category, and show that it is indeed an abelian quotient of the Grothendieck--Teichm\"uller group, thereby showing that a stable version of Kontsevich's conjecture is indeed true. 

\noindent
{\bf \S 3.1 Monoidal automorphisms of the stable symplectic category:} 

\medskip
\noindent
We begin this section by defining the group of monoidal automorphisms of the stable symplectic homotopy category. Recall that this category $h\ess$ is enriched over the category of $\Omega$-module spectra. Given a commutative $\Omega$-algebra $\E$, let $h\ess \wedge_\Omega \E$ denote the category enriched over $\E$-modules constructed by extending coefficients. 

\medskip
\begin{defn}
Define the group of automorphisms of the stable symplectic category as a derived group $G_\Omega$, whose $\E$-points: $G_\Omega(\E)$, for any commutative $\Omega$-algebra $\E$, are defined as the group of monoidal automorphisms of the category $h\ess \wedge_\Omega \E$ that fix each object. We similarly define the group $G_{s\Omega}$ as the group of automorphisms of the oriented stable symplectic homotopy category. 
\end{defn}

\medskip
\noindent
Our first order of business will be to consider certain endomorphisms of $hs\ess \wedge_{s\Omega} \E$ that will eventually turn out to be the elements in the Lie algebra of $G_{s\Omega}(\E)$. We begin with some preliminary constructions that don't require orientability:

\medskip
\noindent
Given a symplectic manifold $M$, recall that $\Omega(M)$ was a defined as a Thom spectrum: $\underline{\mathcal{G}}(M)^{-\zeta}$. Here $ \pi : \underline{\mathcal{G}}(M) \longrightarrow M$ was a principal $\UO$-bundle, supporting a stable real vector bundle $\zeta : \underline{\mathcal{G}}(M) \longrightarrow \BO$ of virtual dimension $m$. The bundle $\pi$ is classified by the map:
\[ \tau(\Omega, M) : M \longrightarrow \BUO. \]

\begin{thm} \label{torsion}
Given an object of $h\mathbb{S}$ represented by a symplectic manifold $(M,\omega)$, then $\tau(M, \Omega)$ factors through the map $\tau(M) : M \longrightarrow \BU$ that classfies the tangent bundle of $M$, followed by the projection map $\BU\longrightarrow \BUO$. Furthermore, the restriction of $\tau(\Omega, M \times \overline{M})$ along the diagonal: $\Delta : M \longrightarrow M \times \overline{M}$ is trivial.
\ \end{thm}

\begin{proof}
Notice that the bundle $\pi : \underline{\mathcal{G}}(M) \longrightarrow M$ was induced from the frame bundle of $M$, with structure group $\U$, along the left action of $\U$ on $\UO$, so we can lift the map $\tau(\Omega, M)$, to $\BU$, to get the factorization through $\tau(M)$. This proves the first part of the claim. Now notice that the restriction of the tangent bundle of $M \times \overline{M}$ along the diagonal $\Delta$ is canonically isomorphic to the complexification of the tangent bundle of $M$. This may be restated as saying that one has a unique lift $\tau(\Delta)$ that makes the following commute:
\[
\xymatrix{
M \ar[d]^{\Delta} \ar[r]^{\tau(\Delta)}  & \BO \ar[d]  \\
M \times \overline{M} \ar[r]^\tau & \BU.}
\]
It follows that $\tau(\Omega, M \times \overline{M})$ is trivial when restricted along $\Delta$. 
\end{proof}

\medskip
\begin{remark} \label{newton}
The projection map $\BU \longrightarrow \BUO$ is injective in cohomology away from two. It is easy to see that its image is generated by classes $d_{2i+1}$ in degree $4i+2$, whose generating function can be expressed in terms of the Chern classes as:
\[ \sum_i d_i = \frac{\sum_i (-1)^i c_i}{\sum_i c_i} \equiv 1 - \sum_i 2c_{2i+1} + \mbox{decomposables}. \]
Similarly, if one considers the inclusion $\BO \longrightarrow \BU$, then this map is injective in homology away from two. If we let $\sum_i b_i$ denote the homogeneous generators in the image of $\sH_\ast(\C\CP^\infty) \subset \sH_\ast(\BU)$, then the image of $\sH_\ast(\BO)$ is the sub (polynomial) algebra generated by classes $a_i$ (away from two), given by:
\[ \sum_i a_i = (\sum_i b_i) (\sum_i (-1)^i b_i) \equiv 1 + \sum_{i>0} 2b_{2i} + \mbox{decomposables}. \]
\end{remark}

\medskip
\begin{defn}
Henceforth, we work in the oriented category $hs\mathbb{S}$, and we assume that $\pi_\ast \E$ is a $\Q$-vector space. Define $\mathcal{P}_{\E}(\BUSO)$ to be the graded $\E^\ast$-submodule of $\tilde{\E}^\ast(\BUSO)$ consisting of primitive elements in the (commutative) Hopf algebra $\E^\ast(\BUSO)$. 
\end{defn}

\begin{thm} \label{Lie1}
$\mathcal{P}_{\E}(\BUSO)$ induces primitive graded endomorphisms of the category $hs\ess \wedge_{s\Omega} \E$ that fix each object. In other words, there is a natural map of graded $\E^\ast$-modules:
\[ \mathcal{P}(\E) : \mathcal{P}_{\E}(\BUSO) \longrightarrow \End(hs \ess \wedge_{s\Omega} \E). \]
Furthermore, the image of $\mathcal{P}(\E)$ is contained in the subgroup of primitive functors, defined as functors $\varphi$ that fix each object, and are additive with respect to the monoidal structure on the morphisms: 
\[ \varphi(X \wedge_{\E} Y) = \varphi(X) \wedge_{\E} Y + X \wedge_{\E} \varphi(Y), \]
for all morphisms $X$ and $Y$. 
\end{thm}
\begin{proof}
Fix objects $(M, \omega)$ and $(N, \eta)$ of $hs\mathbb{S}$. Given an element $\alpha \in \mathcal{P}_{\E}(\BUSO)$, we define the action of $\mathcal{P}(\E)(\alpha)$ on $s\Omega(M, N) \wedge_{s\Omega} \E = s\Omega(\overline{M} \times N) \wedge_{s\Omega} \E := s\Omega(\overline{M} \times N)_{\E}$ as the cap product with $\alpha$, described as the composite map: 
\[ \alpha_\ast : s\Omega(\overline{M} \times N)_{\E} \longrightarrow  s\Omega(\overline{M} \times M)_{\E} \wedge (\overline{M} \times N)_+ \longrightarrow s\Omega(\overline{M} \times N)_{\E} \wedge \BUSO_+ \longrightarrow s\Omega(\overline{M} \times N)_{\E}, \]
where the first map is induced by the Thom diagonal map:
\[ s\underline{\mathcal{G}}(\overline{M} \times N)^{-\zeta} \longrightarrow s\underline{\mathcal{G}}(\overline{M} \times N)^{-\zeta} \wedge (\overline{M} \times N)_+. \]

\noindent
The second map is induced by $\tau(s\Omega, \overline{M} \times N)$, and the third map above is given by capping with the class $\alpha$. Now recall that $s\underline{\mathcal{G}}(\overline{M} \times N)$ is equivalent to the external product bundle $s\underline{\mathcal{G}}(\overline{M}) \times_{\USO} s\underline{\mathcal{G}}(N)$. In particular, the element $\tau(s\Omega,\overline{M} \times N)$ decomposes as the composite:
\[ \overline{M} \times N \longrightarrow \BUSO \times \BUSO \longrightarrow \BUSO. \]

\smallskip
\noindent
Since $\alpha$ is a primitive class, the pullback of $\alpha$ along $\tau(s\Omega,\overline{M} \times N)$ is given by $\alpha_\ast \wedge 1 \, + \, 1 \wedge \alpha_\ast$ under the decomposition $s\Omega(\overline{M} \times N)_{\E} = s\Omega(\overline{M})_{\E} \wedge_{\E} s\Omega(N)_{\E}$.  This argument shows that the elements $\alpha_\ast$ are primitive with respect to the monoidal structure. 

\smallskip
\noindent
Next we will show that the construction $\alpha \mapsto \alpha_\ast$ is functorial. This will yield a map of $\E^\ast$-modules that we seek:
\[ \mathcal{P}(\E) : \mathcal{P}_{\E}(\BUSO) \longrightarrow \End(hs \ess \wedge_{s\Omega} \E). \]

\medskip
\noindent
To show functoriality, we need to show that the diagram:
\[
\xymatrix{
s\Omega(L,M)_{\E} \wedge_{\E} s\Omega(M, N)_{\E} \ar[d] \ar[r]^{\alpha_\ast \wedge \alpha_\ast}  & s\Omega(L, M)_{\E} \wedge_{\E} s\Omega(M, N)_{\E} \ar[d]  \\
s\Omega(L, N)_{\E} \ar[r]^{\alpha_\ast} & s\Omega(L, N)_{\E},}
\]

\medskip
\noindent
commutes; where the vertical maps are induced by composition in $sh\mathbb{S}$, and the top horizontal map: $\alpha_\ast \wedge \alpha_\ast : s\Omega(L, M)_{\E} \wedge_{\E} s\Omega(M, N)_{\E} \longrightarrow s\Omega(L, M)_{\E} \wedge_{\E} s\Omega(M, N)_{\E}$ denotes the external smash product of the two maps $\alpha(\overline{L} \times M)_\ast$ and $\alpha(\overline{M} \times N)_\ast$, where we write $\alpha(X)_\ast$ to indicate that it is an operator on $s\Omega(X)_{\E}$. 

\smallskip
\noindent
By the primitivity of $\alpha_\ast$, we write $\alpha(X \times Y)_\ast$ as the sum $\alpha(X)_\ast \wedge 1 \, + \, 1 \wedge \alpha(Y)_\ast$. This decomposition allows us to reduce the general case to the special case when $L$ and $N$ are a point. In other words, we would like to show that the following special case of the above diagram commutes:
\[
\xymatrix{
s\Omega(\ast, M)_{\E} \wedge_{\E} s\Omega(M, \ast)_{\E} \ar[d] \ar[r]^{\alpha_\ast \wedge \alpha_\ast}  & s\Omega(\ast, M)_{\E} \wedge_{\E} s\Omega(M, \ast)_{\E} \ar[d]  \\
\E \ar[r]^{0} & \E.}
\]

\medskip
\noindent
To show this, recall that the composition $s\Omega(\ast, M)_{\E} \wedge_{\E} s\Omega(M, \ast)_{\E} \longrightarrow \E$ is obtained by restricting along the diagonal $\Delta : M \longrightarrow M \times \overline{M}$. By claim \ref{torsion}, we see that the restriction of $\tau(s\Omega, M \times \overline{M})$ along $\Delta$ is trivial. It follows that the $\alpha_\ast \wedge \alpha_\ast$ followed by composition is trivial. 
\end{proof}

\noindent
{\bf \S 3.2 The structure of primitives:} 

\medskip
\noindent
Let $\mathcal{P}_{\E}(\BU) \subset \tilde{\E}^\ast(\BU)$ denote the submodule of primitives. For complex oriented theories it is a standard fact that this is a free (completed) $\E^\ast$-module generated by the Newton polynomials $N_k(c_1, \ldots, c_k)$, in the Chern classes. These Newton polynomials $N_i(\sigma_1, \cdots, \sigma_i)$ are defined by writing the power symmetric functions $x_1^i + x_2^i + \cdots + x_i^i$ in terms of the elementary symmetric functions $\sigma_1, \sigma_2, \cdots, \sigma_i$. 

\medskip
\noindent
Rationally, the classes $N_k$ can be expressed in terms of the Chern classes $c_k$, or the classes $d_k$ of remark \ref{newton}, by comparing the homogeneous terms in the formal graded equalities (see \cite{McD} Ch.1):
\[ \sum_{k \geq 0} c_k = \prod_{i \geq 0} \exp{\frac{(-1)^i N_{i+1}}{i+1}}, \quad \quad \sum_{k \geq 0} d_k = \prod_{i \geq 0} \exp{\frac{-2N_{2i+1}}{2i+1}}. \]
Up to a scaling factor of $k!$, $N_k(c_1,\ldots, c_k)$ is the homogeneous degree $2k$ term in the Chern character $ch_k$ for the universal virtual vector bundle over $\BU$. Notice that any theory $\E$ such that $\E_\ast$ is a $\Q$-vector space, is complex orientable. 

\begin{thm} \label{Lie2}
Assume that $\E_\ast$ is a $\Q$-vector space. Fix a complex orientation on $\E$ (see remark \ref{orient} below). Then, in cohomology, the map induced by the projection: \[ \mathcal{P}_{\E}(\BUSO) \longrightarrow \mathcal{P}_{\E}(\BU)\] is injective onto the free (completed) $\E^\ast$-module generated by the primitives $N_{2k+1}(c_1, \ldots, c_{2k+1})$ in degree $4k+2$, with $k \geq 0$. In particular, the image of $\mathcal{P}(\E)$ is generated by operators acting on $s\Omega(M)$ via multiplication with the classes $ch_{2k+1}(\tau)$ (compare with \cite{KO} Theorem 9). 
\end{thm}
\begin{proof}
Since $\E_\ast$ is a $\Q$-vector space, we may assume $\E$ is a generalized Eilenberg--MacLane spectrum. 
Notice that the projection map $\BU \longrightarrow \BUSO$ is a map of H-spaces, and in homology, it maps the indecomposable elements in degrees $4k+2$ isomorphically onto the indecomposables in the Hopf-algebra $\E_\ast(\BUSO)$ (see remark \ref{newton}). Dually, it follows that $\mathcal{P}_E(\BUSO)$ maps isomorphically to the completed subspace generated by primitives in degree $4k+2$ in $\mathcal{P}_{\E}(\BU)$. By the above discussion, this is the completed subspace generated by the odd Newton polynomials in the Chern classes. 
\end{proof}

\medskip
\noindent

\begin{corr} \label{expo} Working in characteristic zero, let $G(\E)$ denote the (pro) abelian group generated by the formal exponentials of the form $\Exp(t \, ch_{2k+1}(\tau))$, with $t$ being any homogeneous element of degree $4k+2$ in the $\Q$-vector space $\E_\ast$, where $k \geq 0$. Then $G(\E)$ acts by degree-preserving monoidal automorphisms on the category $sh\ess \wedge_{s\Omega} \E$ that fix each object. In particular, $G(\E)$ is a subgroup of the Galois group of automorphisms $G_{s\Omega}(\E)$. Presently, we will describe a canonical integral form for this group. 
\end{corr}

\begin{remark} \label{orient}
The reader may wish to verify that the completed subspace generated by the odd Newton polynomials in the Chern classes is independent of the choice of complex orientation. In particular, the same holds for the group $G(\E)$. The reader may also verify that the action of $G(\E)$ restricts to the identity on the subgroup: $\pi_0(s\Omega(M,N)) \subset \pi_\ast (s\Omega(M,N) \wedge_{s\Omega} \E)$. 
\end{remark}

\begin{remark} \label{mirror symmetry}
Given a Calabi-Yau manifold $M$, homological mirror symmetry gives rise to an action of the Grothendieck--Teichm\"uller group on $\mbox{H}^*(M, \C)$ under the identification of $\mbox{H}^*(M, \C)$ with the cohomology of poly-vector fields on $M$ (see \cite{KO}, Theorem 9). Remarkably, this action is exactly the same as that of $G(\HC)$ on $s\Omega(M)_{\HC}$, once we use the Thom isomorphism to identify $s\Omega(M)_{\HC}$ with $\mbox{H}^*(M,\C)$. It would be very interesting to give a geometric description of this identification. Also see remark \ref{mirror symmetry2}.
\end{remark}

\medskip
\noindent
{\bf \S 3.3 An integral candidate for the (abelianized) Grothendieck--Teichm\"uller group:} 

\medskip
\noindent
In theorems \ref{Lie1} and \ref{Lie2} we described the Lie algebra of the group $G(\E)$ as being the primitives in $\tilde{\E}^\ast(\BUSO)$. This implies that the cotangent space of $G(\E)$ at the identity element should be interpreted as the vector space dual to these primitives. This dual space can be canonically identified with a subspace of the  indecomposables: $\De(\E_\ast(\BUSO)) = \mbox{I}/\mbox{I}^2 $, where $\mbox{I}$ is the augmentation ideal in $\E_\ast(\BUSO)$. This suggests that one must think of the commutative ring spectrum $s\Omega \wedge \BUSO_+$ as ``functions on a derived avatar" of the abelianized Grothendieck--Teichmuller group. 

\smallskip
\noindent
The spectrum $s\Omega \wedge \BUSO_+$ can be constructed functorially from $s\Omega$: Indeed, we know by \cite{B} (Prop. 7.3), that $s\Omega \wedge \BUSO_+$ is equivalent to $\THH(s\Omega)$ as commutative algebras. Notice in fact, that $s\Omega \wedge \BUSO_+$ is a commutative Hopf-algebra spectrum in the category of $s\Omega$-module spectra \cite{AR}. The above discussion leads naturally to the:
\begin{defn} \label{GTgroup}
Define a derived group scheme ${\bf G} = \mbox{Spec} \, \THH(s\Omega)$, whose $\E$-points for an arbitrary commutative $s\Omega$-algebra $\E$ is defined to be the group of homotopy classes of $s\Omega$-algebra maps from $\THH(s\Omega)$ to $\E$:
\[ {\bf G}(\E) = \Alg_{s\Omega}(\THH(s\Omega), \E) = \Alg_{\Sph}(\BUSO_+, \E). \]
\end{defn}

\begin{thm} \label{thh}
Given a commutative $s\Omega$-algebra $\E$, the group ${\bf G}(\E)$ acts by degree preserving monoidal automorphisms on the category $hs\ess \wedge_{s\Omega} \E$ that fix each object. 
\end{thm}
\begin{proof}
The proof of this theorem is similar to the proof of theorem \ref{Lie1}. As before, given $\beta \in {\bf G}(\E)$, we get an automorphism of $s\Omega(\overline{M} \times N)_{\E}$ given by the cap product with $\beta$:
\[ \beta_\ast : s\Omega(\overline{M} \times N)_{\E} \longrightarrow  s\Omega(\overline{M} \times N)_{\E} \wedge (\overline{M} \times N)_+ \longrightarrow s\Omega(\overline{M} \times N)_{\E} \wedge \BUSO_+ \longrightarrow s\Omega(\overline{M} \times N)_{\E}. \]
Since $\beta$ is a map of algebras, we see that $\beta_\ast$ preserves the monoidal structure. So the only thing left to check is that $\beta_\ast$ is functorial. As in the proof of theorem \ref{Lie1}, we may reduce this question to showing that the following diagram commutes:
\[
\xymatrix{
s\Omega(\ast, M)_{\E} \wedge_{\E} s\Omega(M, \ast)_{\E} \ar[d] \ar[r]^{\beta_\ast \wedge \beta_\ast}  & s\Omega(\ast, M)_{\E} \wedge_{\E} s\Omega(M, \ast)_{\E} \ar[d]  \\
\E \ar[r]^{=} & \E.}
\]
Again, as in the proof of \ref{Lie1}, this requires showing that the restriction of $\tau(s\Omega, M \times \overline{M})$ along $\Delta$ below, factors through the unit of $\E$:
\[ M_+ \longrightarrow M_+ \wedge \overline{M}_+ \longrightarrow \BUSO_+ \longrightarrow \E. \]
But this follows from theorem \ref{torsion}, and the fact that $\beta$ is a ring map. 
\end{proof}

\noindent
{\bf \S 3.4 The (abelianized) Grothendieck--Teichm\"uller group as a Motivic group:}

\medskip
\noindent
Given a symplectic manifold $(M,\omega)$ recall that the morphism spectrum $\Omega(\ast, M)$ in $h\ess$ can be identified with the $\Omega$-module spectrum $\Omega(M)$. Let $\pi_0 (\ess)$ denote the truncation of $\ess$, so that the objects of $\pi_0(\ess)$ are the same as those of $h\ess$, and morphisms given by $\pi_0 (\Omega(M,N))$. Given a commutative $\Omega$-algebra $\E$, right composition in $h\ess$ gives rise to a functor $\mathcal{F}_{\E}$ with values in $\E$-module spectra:
\[ \mathcal{F}_{\E} :  \pi_0(\mathbb{S}) \longrightarrow \E \mathcal{S}, \quad \quad \mathcal{F}_{\E}(M) = \Omega(M)_{\E} := \Omega(M) \wedge_{\Omega} \E,\] 
where $\E \mathcal{S}$ denotes the homotopy category of $\E$-module spectra. Recall that the category $h\mathbb{S}$ is a symmetric-monoidal category, with the monoidal structure given by the cartesian product of symplectic manifolds. Since $\Omega(M \times N)$ is equivalent to $\Omega(M) \wedge_\Omega \Omega(N)$, the functor $\mathcal{F}_{\E}$ is monoidal. 

\medskip
\noindent
Constructions analogous to the ones described above can be made in oriented stable symplectic homotopy category: $hs\ess$. Furthermore, these constructions remain nontrivial when tensored with $\Q$. Henceforth, we will work over $\Q$. Now, given a commutative $s\Omega$-algebra $\E$, and an element $\beta \in G(\E)$ as defined in corollary \ref{expo}, recall that we have an $\E$-module automorphism: 
\[ \beta_\ast(M) : s\Omega(M)_{\E} \longrightarrow s\Omega(M)_{\E}. \]
Furthermore, $\beta_\ast$ is monoidal, i.e. $\beta_\ast(M \times N) = \beta_\ast(M) \wedge \beta_\ast(N)$ under the monoidal structure of $hs\ess$. In addition, from remark \ref{orient}, we know that composition induces a map: 
\[ \mathcal{F}_{\E} : \pi_0(s\Omega(M,N)) \longrightarrow [s\Omega(M)_{\E}, s\Omega(N)_{\E}], \]
lands inside operators that commute with $\beta_\ast$. In other words, we see that $\beta_\ast$ is a natural automorphism of the functor $\mathcal{F}_{\E}$. In particular, we see that the abelianized Grothendieck--Teichm\"uller groupscheme is a natural sub groupscheme of the motivic groupscheme of monoidal automorphisms of the functor $\E \longmapsto \mathcal{F}_{\E}$.

\begin{remark}\label{rigid}
Recall that the map: 
\[ \mathcal{F}_{\E} : \pi_0(s\Omega(M,N)) \longrightarrow [s\Omega(M)_{\E}, s\Omega(N)_{\E}] \]
is an equivalence if $M$ is a compact manifold. In particular, $\mathcal{F}_{\E}$ is a rigid monoidal functor on the subcategory of $\pi_0(\mathbb{S})$ generated by compact symplectic manifolds $(M, \omega)$. 
\end{remark}

\begin{remark}
Consider the map $\BU \longrightarrow \THH(s\Omega)$ induced by the inclusion of the unit $\Sph \longrightarrow s\Omega$: 
\[ \BU \longrightarrow \BUSO_+ \longrightarrow s\Omega \wedge \BUSO_+ = \THH(s\Omega). \]
This map may be shown to factor through $\BU \longrightarrow \mbox{K}(s\Omega)$ lifting the Dennis trace map. This suggests a close relation between the Waldhausen $K$-theory of $s\Omega$ and the pro-abelian group scheme {\bf G}. This is strikingly reminiscent of Kato's higher classfield theory \cite{Kato}[Theorem 2.1], which relates the algebraic K-theory of higher local fields to the Galois groups of their abelian extensions.\end{remark}

\medskip
\noindent
{\bf \S 3.5 Final remarks and speculation}

\bigskip

\noindent
In the sections that follow, we borrow notation from \cite{M1}, where $\Sph[\Gr_+]$ denotes the suspension spectrum of $\Gr$, viewed as a kind of group ring, for an $H$-space $\Gr$. 

\medskip
\noindent
Kontsevich's 1999 paper suggests that an action of $\GT_\Q$ defines a deformation of the complexified $\hat{A}$-genus, which can be interpreted as associated to the formal group law with $\Gamma(x)^{-1}$ as its exponential (where $\Gamma(x)$ denotes the Gamma-function, \cite{M1}[\S 2.3.1]): We seek to generalize Kontsevich's result to our context. Let $\MU \longrightarrow \MSO$ denote the forgetful map, and let $\MU \longrightarrow \Sph[\BU_+] \wedge \MU$ denote the diagonal map. Then given a commutative ring spectrum $\E$ with an $\SO$-orientation $\rho : \MSO \longrightarrow \E$, we have a map of ring spectra $\Gamma_\rho$:
\[ \Gamma_\rho : \MU \longrightarrow \Sph[\BU_+] \wedge \MU \longrightarrow  \Sph[\BUSO_+] \wedge \MSO \longrightarrow \Sph[\BUSO_+] \wedge \E, \]
Note that the map $\Gamma_\rho$ above can be expressed as a morphism:
\[
\Gamma_\rho : \MU \longrightarrow \THH(s\Omega) \wedge_{s\Omega} \E.
\] 
Using this description, one may generalize Kontsevich's construction in our framework: namely, one may define a torsor of deformations of the $\rho$-orientation under the action of the group ${\bf G}(\E)$ as follows: Given an element in ${\bf G}(\E)$ represented by a ring map: $\beta : \BUSO_+ \longrightarrow \E$, the corresponding deformation of $\rho$ is given by capping $\beta$ with $\Gamma_\rho$:
\[ \rho_\beta = \beta \cap \Gamma_\rho : \MU \longrightarrow \Sph[\BUSO_+] \wedge \E \longrightarrow \E. \]

\smallskip
\noindent
There is a metaplectic analog of this whole picture. Recall (\cite{N}, \S8), that the unitary group $\U$ admits a natural double cover $\tU$ that supports the square-root of the determinant homomorphism. The forgetful map $\U \longrightarrow \SO$ lifts to a unique map $\tU \longrightarrow \Spin$. Therefore, given a spectrum $\E$ with a  $\Spin$-orientation $\rho : \MSpin \longrightarrow \E$, we have the corresponding:
\[
\tilde{\Gamma}_\rho : \MtU \longrightarrow \THH(s\tilde{\Omega}) \wedge_{s\tilde{\Omega}} \E,
\] 
leading to a ${\bf G}(\E)$-torsor of deformations of $\rho$ as before. 

\begin{remark}
Notice that the maps: 
\[ \MtU \longrightarrow \Sph[\BtUSpin_+] \wedge \MSpin, \quad \quad \MU \longrightarrow \Sph[\BUSO_+] \wedge \MSO \]
 are equivalences away from the prime two. In addition, the spaces $\BtUSpin$, $\BUSO$ and $(\SpU)$ are also equivalent away from two. 
\end{remark}
\medskip
\noindent
An example of such a deformation is the one suggested by Kontsevich above: Let $\E$ denote complex K-theory with complex coefficients: $\KU \otimes \, \C$ supporting the $\hat{A}$-orientation (or rather, its scalar extension by $\C$): $\hat{A} : \MSpin \longrightarrow \KU \otimes \, \C$. 
In \cite{M1}[\S 2.3.1]), we construct a deformation $\hat{A}_\zeta$ by specializing the generators of the polynomial algebra $\sH_\ast(\SpU, \Q)$ to odd zeta-values in $\KU \otimes \, \C$, graded using suitable powers of the Bott class. The genus associated to $\hat{A}_\zeta$ is therefore the composite: 
\[
\hat{A}_{\zeta} : \MU_*  \longrightarrow \sH_*(\SpU_+,\KU \otimes \, \Q) \longrightarrow \KU_* \otimes \, \C 
\]
which restricts to the $\hat{A}$-genus on $\MSp$ (which is $\MSO$, or $\MSpin$ away from two), and sends manifolds of dimension $\equiv 2$ mod 4 to $i\R \subset \C$; more precisely, the primitives of $\sH_*(\SpU_+,\Q)$, interpreted as symmetric functions, are sent (as explained in the appendix below) to odd zeta-values, graded using the Bott class.

\medskip
\noindent
Let us now describe the geometry behind this deformation induced by the group ${\bf G}(\E)$. First consider the spectrum $\BUSO_+ \wedge \MSO$. Using the language used in Section \S2, we may identify an element in $\pi_k(\BUSO_+ \wedge \MSO) = \pi_k(\THH(s\Omega) \wedge_{s\Omega} \MSO)$ as the stabilization (with respect to the integer $n$) of the data given by a concordance class of oriented manifolds embedded in euclidean space: $M^{k+n} \subset \R^\infty \times \R^n$, that are proper over $\R^n$ and are endowed with a principal bundle of oriented lagrangian grassmannians which is classified by a map $\theta : M^{k+n} \longrightarrow \BUSO$. In this language, the map $\Gamma_\rho : \pi_k(\MU) \longrightarrow \pi_k(\THH(s\Omega) \wedge_{s\Omega} \MSO)$ described above identifies a stably almost complex manifold $M^k$ with the underlying oriented manifold, endowed with the formal negative of the bundle $\xi$, where $\xi : s\underline{\mathcal{G}}(M) \longrightarrow M$ is the bundle of oriented lagrangians in the tangent bundle (see \S2). Given an $\E$-valued genus $\rho$, the action of ${\bf G}(\E)$, in the language used above, corresponds to deforming $\rho$ along different choices of multiplicative $\E$-theory characteristic classes of the bundle $\theta$. The analogous statements hold in the metaplectic case. 

\medskip
\noindent
\begin{remark}
It is a compelling question to ask about the meaning of these (lagrangian-bundle) deformations in terms of elliptic differential operators (like the Dirac operator) that define genera. This framework bears a striking resemblance to the construction of the analytic torsion classes \cite{BL} and one would like to know if there is any relation. 
\end{remark}

\section{{Appendix}: Grothendieck-Teichm\"uller groups}\label{hotGT}

\noindent
We recall some of the complex history \cite{YA}(\S 25.9) of this subject, which has deep connections to homotopy theory but which may be unfamiliar. 

\medskip
\noindent
For topologists, the operad defined by Artin's braid groups is a good place to start: it has few automorphisms, but its completions are less rigid. Ihara has studied its system of profinite completions, but for our purposes Drinfeld's work \cite{DQ} on the automorphism group of its system of Malcev $\Q$-completions will be more relevant. Kontsevich compared the latter object to the group of homotopy automorphisms of the rational chains on the little disk operad \cite{Fr}, and suggested \cite{KO}[\S 4.4] that both these objects are isomorphic to the motivic Galois group of a certain \cite{DG} Tannakian category of mixed Tate motives. \medskip

\noindent
We follow Kontsevich's lead here, and refer to all these conjecturally equivalent pro-algebraic $\Q$-groupschemes as {\bf the} Grothendieck-Teichm\"uller group $\GT_\Q$; we take it to be an extension of the form:
\[
1 \to \mF \longrightarrow \GT_\Q \longrightarrow \Gm \to 1
\]
with $\mF$ pro-unipotent, defined by a graded Lie algebra $\mf$ free on generators $z_{2i+1}, \; i>0$, associated to the generators of the rank one abelian groups $\mbox{K}^\alg_{4n+1}(\Z)$ through the manifestation of $\GT_\Q$ as the motivic group of a $\Q$-linear category $\MTM_{\Z}$ of mixed Tate motives, generated by certain cell-like objects $\Z(n)$ satisfying
\[
\Ext^*_{\MTM}(\Z(0),\Z(n)) = \Q \; {\rm if} \;  * = 0 \; {\rm and} \; n = 0
\]
\[
\; = \; \mbox{K}^\alg_{2n+1}(\Z) \otimes \Q \; {\rm if} \; * = 1,
\]
these groups being zero otherwise. The derived group $\mF$ of $\GT_\Q$ has, as its abelianization, a groupscheme represented by a commutative and cocommutative Hopf algebra over $\Q$, with primitives in topological degree 4k+2, corresponding to $\sH_*(\SpU;\Q)$. Borel's work on regulators thus relates the generators $z_{2i+1}$ to the (conjecturally transcendental) zeta-values $\zeta(2i+1)$. 
\medskip

\noindent
The affine groupscheme $\mF$ has a $\Gm$-action, and the associated commutative Hopf algebra of functions $\Qsymm \otimes \Q$ is dual to the universal enveloping algebra of its Lie algebra $\mf$. That is the free associative $\Q$-algebra:
\[
U(\mf) = \Q \langle \langle z_{2i+1} \:|\: i > 0 \rangle \rangle = \Nsymm
\otimes \Q
\]
of rational noncommutative symmetric functions \cite{BR,CPrim}, with diagonal defined by the juxtaposition coproduct. The (self-dual) Hopf algebra $\Symm$ of classical symmetric functions is dual to the abelianization of $\Nsymm$: the abelianization $\mf \to \mf_\ab$ defines a homomorphism $U(\mf) \to U(\mf_\ab)$ dual to the inclusion of the symmetric functions in the quasisymmetric ones, defining a quotient 
\[
1 \to \mF_\ab \longrightarrow \widetilde{\GT}_\Q \longrightarrow \Gm \to 1
\]
of $\GT_\Q$ with $\mf_\ab$ as its (graded abelian) Lie algebra. \medskip

\noindent
One consequence of deep work of Hatcher, Waldhausen, B\"okstedt, Rognes and others is a rational equivalence \cite{M2}(\S 3)
\[
(\Sph \vee \Sigma \kO)_\Q \longrightarrow \mbox{K}(\Sph)_\Q
\]
of spectra, which yields a canonical identification 
\[
\Sph_\Q \wedge^L_{K(\Sph)} \Sph_\Q \; \cong \; \Qsymm_\Q
\]
of the rational covariant Koszul dual of Waldhausen's K-theory of the sphere spectrum with a model for the Hopf algebra of functions on $\GT_\Q$. The analogous Koszul dual of the stabilization morphism
\[
\Sph[\SUSO] \longrightarrow \Sph \wedge \Sigma \kO
\]
of ringspectra defines a morphism
\[
\Sph_\Q \wedge^L_{\Sph[\SUSO]} \Sph_\Q \; \cong \; \Symm_\Q \to \Qsymm_\Q
\]
representing the quotient $\GT_\Q \to \widetilde{\GT}_\Q$. \medskip

\noindent
The map that sends the Newton's power sums
\[
N_n \; := \: \sum_{k \geq 1} x_k^n \in \Symm \subset \Qsymm
\]
to the real number given by the Riemann-zeta value $\zeta(n)$, under the specialization that sends $x_k \mapsto 1/k$ can, with some care, be extended from $\Symm$ to a ring homomomorphism:
\[
\bzeta : \Qsymm \longrightarrow \MZV \subset \R
\]
with values in a certain graded algebra of real multizeta values\begin{footnote}{The case $n=1$, ie $i=0$, which needs special treatment \cite{CB}, is interestingly absent from the constructions considered above.}\end{footnote}.\medskip

\noindent
Remarkably enough, these multizeta values play an important role in Connes,Kreimer, and Marcolli's Galois-theoretic reinterpretation \cite{CM1,CM2} of the classical BPS renormalization theory of Feynman integrals, in which MZVs appear ubiquitously in explicit computations. Kontsevich found an action \cite{KO}[\S 4.6 Th 9] of the abelianization of $\GT_\Q$ on a moduli space for deformation quantizations of Poisson manifolds, through an action of the little disk operad on Hochschild homology. These developments led Cartier \cite{CAMS} to suggest that $\GT_\Q$ is in some sense a {\em cosmic} Galois group.\medskip

\begin{remark} \label{mirror symmetry2}
As noted in Remark \ref{mirror symmetry}, Kontsevich \cite{K1}[\S 4.6.2, 8.4] shows that the graded cohomology algebra
\[
\sH^*(M,\Lambda^*T)
\]
of polyvector fields on a complex manifold is isomorphic to a kind of Hochschild cohomology
\[
\HHh^*(M) \; = \; \Ext^*_{\oh_{M \times M}}(\oh_M,\oh_M)
\]
defined in terms of the coherent sheaf of holomorphic functions on $M \times M$ and its diagonal. If $M$ is Calabi-Yau, we can identify its tangent and cotangent bundle, obtaining an isomorphism of both cohomologies with the Hodge cohomology of $M$, defining  an action of (some version of) the Grothendieck-Teichm\"uller group on the de Rham cohomology of $M$. The homotopy-theoretic point of view suggests the Calabi-Yau hypothesis may be unnecessarily strong.\bigskip
\end{remark}

\pagestyle{empty}
\bibliographystyle{amsplain}

\begin{thebibliography}{10}

\bibitem{YA}
Y. Andr\'e, 
\emph{Une introduction aux motifs \dots}, Panoramas et Synth\'eses 17, Soci\'et\'e Math\'ematique de France (2004)



\bibitem{AR}
C.\ Ausoni and J.\ Rognes, \emph{Rational Algebraic K-theory of topological K-theory}, available at {\tt arXiv:math.AT/0708.2160}.

\bibitem{A}
D.\ Ayala, \emph{Geometric cobordism categories}, available at {\tt arXiv:0811.2208}.


\bibitem{BR}
A.\ Baker, B.\ Richter, \emph{Quasisymmetric functions from a topological point of view.} Math. Scand. 103 (2008) 208 - 242, available at {\tt arXiv:math/0605743}.


\bibitem{B}
A.\ Blumberg, \emph{Topological Hochschild homology of Thom spectra which are $E_\infty$-ring spectra}, available at: arXiv:0811.0803v1, 2008. 



\bibitem{BL}
J-M. \ Bismut and J. \ Lott, \emph{Flat vector bundles, direct images and higher real analytic torsion}, Journal of the American Math. Soc., Vol 8, No. 2, 291--363, 1995. 


\bibitem{CB}
P.\ Cartier, \emph{Fonctions polylogarithmes, nombres polyz\'etas et groupes pro-unipotents}, S\'eminaire Bourbaki, Vol. 2000/2001. Ast\'erisque 282 (2002) Exp. 885, 137 - 173

\bibitem{CAMS}
-----, \emph{A mad day's work \dots}, BAMS 38 (2001) 389 - 408.

\bibitem{CPrim}
------, \emph{A primer of Hopf algebras}, in Frontiers in number theory, physics, and geometry. II, 537 - 615, Springer (2007).

\bibitem{CM1}
A.\ Connes, M.\ Marcolli, \emph{From physics to number theory via noncommutative geometry, Part II: renormalization, the Riemann-Hilbert correspondence, and motivic Galois theory}, available at {\tt arXiv:hep-th/0411114}.

\bibitem{CM2}
-----, -----, \emph{Renormalization and motivic Galois theory},  available at {\tt arXiv:math/0409306}


\bibitem{DG}
------, A.\ Goncharov, \emph{Groupes fondamentaux motiviques de Tate mixte}. Ann. Sci. \'Ecole Norm. Sup. 38 (2005) 1 - 56.

\bibitem{DQ}
V.\ Drinfel'd, \emph{On quasitriangular quasi-Hopf algebras and on a group that is closely connected with} Gal$(\overline{\Q}/\Q)$, Leningrad Math. J. 2 (1991) 829 - 860.




\bibitem{Fr} 
B. \ Fresse, \emph{Homotopy of operads and Grothendieck-Teichm\"uller groups}, available at 
{\tt http://math.univ-lille1.fr/}$\sim${\tt fresse/OperadHomotopyBook/}

\bibitem{GS}
G. Segal, \emph{A geometric perspective on quantum field theory}, Contemp. Math. (AMS), Vol. 620, 281--294, 2014. 







\bibitem{Kato}
K.\ Kato, Generalized class-field theory, available at \[{\tt http://www.mathunion.org/ICM/ICM1990.1/Main/icm1990.1.0419.0428.ocr.pdf}\]

\bibitem{N}
N.\ Kitchloo, \emph{The Stable symplectic category and quantization}, Contemp. Math. (AMS), Vol. 620, 251--280, 2014. available at arXiv:1204.5720. 



\bibitem{K1}
M.\ Kontsevich, \emph{Deformation quantization of Poisson manifolds, I}, Lett. Math. Phys 66 (157-216) 2003, available at {\tt arXiv:q-alg/9709040}.

\bibitem{KO}
-----, \emph{Operads and motives in deformation quantization}, Lett. Math. Phys. 48, 35--72, 1999, available at {\tt arXiv:9904055}. 

\bibitem{McD}
I.\ MacDonald, \emph{Symmetric functions and Hall polynomials}, 2nd ed, Oxford Mathematical Monographs (1995).

\bibitem{M2}
J.\ Morava, \emph{Homotopy-theoretically enriched categories of noncommutative motives}, Research in the Mathematical 
Sciences. 2:8 (2015), available at {\tt arXiv:1402.3693}

\bibitem{M1}
J.\ Morava, \emph{An integral lift of the $\Gamma$-genus}, available at {\tt arXiv:1101.1647}.





\bibitem{W}
A.\ Weinstein, \emph{Symplectic Categories}, available at {\tt arXiv:0911.4133}.


\bibitem{W2}
A.\ Weinstein, \emph{Symplectic Geometry}, Bulletin of the A.M.S., Vol.5, No.1, 1--13, 1981. 


\end{thebibliography}
\providecommand{\bysame}{\leavevmode\hbox
to3em{\hrulefill}\thinspace}

\end{document}